\documentclass[12pt]{amsart}
\usepackage{amsmath,amssymb,amsbsy,amsfonts,latexsym,amsopn,amstext,
                                               amsxtra,euscript,amscd,bm,cite}
\usepackage{url}
\usepackage[colorlinks,linkcolor=blue,anchorcolor=blue,citecolor=blue,backref=page]{hyperref}
\usepackage{color}
\usepackage{float}

\hypersetup{breaklinks=true}

\usepackage[norefs,nocites]{refcheck}
\begin{document}

\newtheorem{theorem}{Theorem}
\newtheorem{lemma}[theorem]{Lemma}
\newtheorem{claim}[theorem]{Claim}
\newtheorem{cor}[theorem]{Corollary}
\newtheorem{prop}[theorem]{Proposition}
\newtheorem{definition}{Definition}
\newtheorem{question}[theorem]{Question}
\newtheorem{remark}[theorem]{Remark}
\newcommand{\hh}{{{\mathrm h}}}

\numberwithin{equation}{section}
\numberwithin{theorem}{section}
\numberwithin{table}{section}

\def\sssum{\mathop{\sum\!\sum\!\sum}}
\def\ssum{\mathop{\sum\ldots \sum}}
\def\iint{\mathop{\int\ldots \int}}

\def\squareforqed{\hbox{\rlap{$\sqcap$}$\sqcup$}}
\def\qed{\ifmmode\squareforqed\else{\unskip\nobreak\hfil
\penalty50\hskip1em\null\nobreak\hfil\squareforqed
\parfillskip=0pt\finalhyphendemerits=0\endgraf}\fi}

\newfont{\teneufm}{eufm10}
\newfont{\seveneufm}{eufm7}
\newfont{\fiveeufm}{eufm5}
%
%
\newfam\eufmfam
     \textfont\eufmfam=\teneufm
\scriptfont\eufmfam=\seveneufm
     \scriptscriptfont\eufmfam=\fiveeufm
%
%
\def\frak#1{{\fam\eufmfam\relax#1}}

\newcommand{\bflambda}{{\boldsymbol{\lambda}}}
\newcommand{\bfmu}{{\boldsymbol{\mu}}}
\newcommand{\bfxi}{{\boldsymbol{\xi}}}
\newcommand{\bfrho}{{\boldsymbol{\rho}}}

\newcommand{\bfalpha}{{\boldsymbol{\alpha}}}
\newcommand{\bfbeta}{{\boldsymbol{\beta}}}
\newcommand{\bfphi}{{\boldsymbol{\varphi}}}
\newcommand{\bfpsi}{{\boldsymbol{\psi}}}
\newcommand{\bftheta}{{\boldsymbol{\vartheta}}}

\def\fK{Frak K}
\def\fT{Frak{T}}

\def\fA{{Frak A}}
\def\fB{{Frak B}}
\def\fC{\mathfrak{C}}

\def \balpha{\bm{\alpha}}
\def \bbeta{\bm{\beta}}
\def \bgamma{\bm{\gamma}}
\def \blambda{\bm{\lambda}}
\def \bchi{\bm{\chi}}
\def \bphi{\bm{\varphi}}
\def \bpsi{\bm{\psi}}

\def\eqref#1{(\ref{#1})}

\def\vec#1{\mathbf{#1}}


\def\cA{{\mathcal A}}
\def\cB{{\mathcal B}}
\def\cC{{\mathcal C}}
\def\cD{{\mathcal D}}
\def\cE{{\mathcal E}}
\def\cF{{\mathcal F}}
\def\cG{{\mathcal G}}
\def\cH{{\mathcal H}}
\def\cI{{\mathcal I}}
\def\cJ{{\mathcal J}}
\def\cK{{\mathcal K}}
\def\cL{{\mathcal L}}
\def\cM{{\mathcal M}}
\def\cN{{\mathcal N}}
\def\cO{{\mathcal O}}
\def\cP{{\mathcal P}}
\def\cQ{{\mathcal Q}}
\def\cR{{\mathcal R}}
\def\cS{{\mathcal S}}
\def\cT{{\mathcal T}}
\def\cU{{\mathcal U}}
\def\cV{{\mathcal V}}
\def\cW{{\mathcal W}}
\def\cX{{\mathcal X}}
\def\cY{{\mathcal Y}}
\def\cZ{{\mathcal Z}}
\newcommand{\rmod}[1]{\: \text{mod} \: #1}

\def\cg{{\mathcal g}}

\def\vr{\mathbf r}

\def\e{{\mathbf{\,e}}}
\def\ep{{\mathbf{\,e}}_p}
\def\em{{\mathbf{\,e}}_m}

\def\Tr{{\mathrm{Tr}}}
\def\Nm{{\mathrm{Nm}\,}}

 \def\SS{{\mathbf{S}}}

\def\lcm{{\mathrm{lcm}}}
\def\ord{{\mathrm{ord}}}

\def\({\left(}
\def\){\right)}
\def\fl#1{\left\lfloor#1\right\rfloor}
\def\rf#1{\left\lceil#1\right\rceil}

\def\mand{\qquad \text{and} \qquad}

\newcommand{\commM}[1]{\marginpar{%
\begin{color}{red}
\vskip-\baselineskip 
\raggedright\footnotesize
\itshape\hrule \smallskip M: #1\par\smallskip\hrule\end{color}}}

\newcommand{\commI}[1]{\marginpar{%
\begin{color}{magenta}
\vskip-\baselineskip 
\raggedright\footnotesize
\itshape\hrule \smallskip I: #1\par\smallskip\hrule\end{color}}}

\newcommand{\commK}[1]{\marginpar{%
\begin{color}{blue}
\vskip-\baselineskip 
\raggedright\footnotesize
\itshape\hrule \smallskip K: #1\par\smallskip\hrule\end{color}}}




\hyphenation{re-pub-lished}

\mathsurround=1pt

\def\bfdefault{b}
\overfullrule=5pt

\def \F{{\mathbb F}}
\def \K{{\mathbb K}}
\def \N{{\mathbb N}}
\def \Z{{\mathbb Z}}
\def \Q{{\mathbb Q}}
\def \R{{\mathbb R}}
\def \C{{\mathbb C}}
\def\Fp{\F_p}
\def \fp{\mathfrak p}
\def \fq{\mathfrak q}

\def\ZK{\Z_K}

\def \xbar{\overline x}
\def\e{{\mathbf{\,e}}}
\def\ep{{\mathbf{\,e}}_p}
\def\eq{{\mathbf{\,e}}_q}


\title[Squarefree  and squarefull integers in   progressions]{Smooth squarefree  and squarefull integers in arithmetic progressions}
\date{\today}

\author[M. Munsch]{Marc Munsch}
\address{5010 Institut f\"{u}r Analysis und Zahlentheorie
8010 Graz, Steyrergasse 30, Graz}
\email{munsch@math.tugraz.at}

 \author[I. E. Shparlinski] {Igor E. Shparlinski}

\address{Department of Pure Mathematics, University of New South Wales,
Sydney, NSW 2052, Australia}
\email{igor.shparlinski@unsw.edu.au}

\author[K. H. Yau]{Kam Hung Yau}

\address{Department of Pure Mathematics, University of New South Wales,
Sydney, NSW 2052, Australia}
\email{kamhung.yau@unsw.edu.au}

\begin{abstract}   
We obtain new lower bounds on the number of smooth squarefree integers up to $x$
in residue classes modulo a prime $p$, relatively large compared to $x$, which in some ranges of $p$ and $x$ improve 
that of A. Balog and C.  Pomerance (1992). We also obtain an estimate on 
the smallest squarefull number in almost all residue classes modulo a prime $p$.
\end{abstract}


\maketitle

\section{Introduction}

\subsection{Background and motivation}

Let $p$ be a prime. For any integer $n \ge 2$ we denote $P^+(n)$ to be the largest prime factor of $n$. For any positive real number $y$, we say that an integer is $y$-smooth if $P^+(n) \le y$.

Studying the distribution of $y$-smooth numbers $n \le x$ in progressions modulo an 
integer $q\ge 2$ has always been a very active subject of research, see~\cite{BaPom,Drap,Harp,MuSh,Sound} and references therein. For instance, as pointed out in ~\cite{Sound}, a very good level of distribution would imply the truth of Vinogradov's conjecture about the smallest quadratic non-residue.

As usual, we denote by $\psi(x,y;p,a)$  the number of positive integers $n \le x$ 
which are $y$-smooth and satisfy $n \equiv a \bmod{p}$. Furthermore, we use $\psi^\sharp(x,y;p,a)$ for the number of those integers which are also squarefree.

Due to its link with Euclidean prime generators, the positivity of $\psi^\sharp(x,y;p,a)$  in the special case of $y=p$ is of special interest, see~\cite{BoPom}. Thus, following  Booker and Pomerance~\cite{BoPom}, 
we   use $M(p)$ to
 denote  the least $x$ such that  $\psi^\sharp(x,p;p,a)>0$ for every integer $a$. The quantity $M(p)$  has been considered in~\cite{MuSh}, where in particular the 
conjecture of Booker and Pomerance~\cite{BoPom} that $M(p)= p^{O(1)}$ is established in 
a stronger form
$$
M(p) \le p^{3/2+o(1)}, 
$$
for all primes $p$,  and 
$$
M(p) \le p^{4/3+o(1)},
$$
for all, but a set of primes $p$ of relative zero density.

Here we use similar ideas to obtain lower bounds on $\psi^\sharp(x,y;p,a)$ 
of essentially the right order of magnitude in a broader range of $y$. 
These bounds, even without taking into account the squarefreeness condition, 
that is, using
$$
\psi(x,y;p,a) \ge \psi^\sharp(x,y;p,a), 
$$
improve the range in which the result of Balog and  Pomerance~\cite{BaPom}
applies. 

Subsequently, we also address a question about squarefull 
numbers in arithmetic progressions (that is numbers, which are divisible by
squares of all their prime divisors). 
This question is significantly less studied, see however~\cite{Chan,Mun,MuTr}.
In particular, Chan~\cite{Chan} obtained an asymptotic formula for the number of squarefull 
numbers in an arithmetic progression, however, due to a rather complicated structure of the 
main term, it is not immediately clear when the main term starts to exceed the error term. Here we consider a Linnik-type version of this question. 
Namely, using very different arguments compared 
to the case of squarefree numbers (and also to~\cite{Chan}), we investigate the quantity $F(a,p)$ 
which is defined as the smallest positive squarefull number $n \equiv a \bmod p$.

We note that the question about squarefull numbers in arithmetic progressions is dual to the question
on squarefull, and more generally $k$-full numbers (that is, numbers divisible by $k$-th power of all their prime divisors)
 in short intervals, which has been considered 
in~\cite{KoLu,KoLuSh}.  In particular, it is shown in~\cite[Theorem~1]{KoLuSh} that infinitely many intervals 
of the form $(N^{k}, (N+1)^k)$ contain at least 
$$
M \ge \( \(\frac{3}{8} + o(1)\) \frac{\log N}{\log \log N}\)^{1/3}
$$
$k$-full integers (but of course no perfect $k$-th powers).  Here we use an opportunity to present in Appendix~\ref{app:A} an 
argument of V.~Blomer which allows to replace $1/3$ with $1/2$ in the lower bound on $M$. 

\subsection{Results for squarefree numbers}

We  start with a lower bound on  $\psi^\sharp(x,y;p,a)$
which holds for any prime $p$.

\begin{theorem} \label{thm: smooth sqfr}
For any fixed real numbers $\alpha$ and $\beta$ with $\beta  \in (23/24,1]$ and $\alpha \in (9/2 -3\beta,  3\beta]$,
for $x = p^{\alpha+ o(1)}$ and $y=p^{\beta+ o(1)}$ we have
$$
 \psi^\sharp(x,y;p,a) \ge x^{1+o(1)}/p
$$
as $p \rightarrow \infty$.
\end{theorem}

Taking $y=p^{\beta}$ with $23/24<\beta \le 1$ and $q=p$ in the main result of Balog \& Pomerance~\cite{BaPom} gives the existence of a $p^{\beta}$-smooth integer (not necessary squarefree) $n \le p^{\max \{ 3\beta /2, 3/4 +\beta \} +o(1)} =p^{3/4 +\beta  +o(1)} $ since $\beta \le 1$. We notice that
$$
9/2 - 3\beta  < 3/4 +\beta,
$$ under the condition $23/24 < \beta$. Therefore, Theorem~\ref{thm: smooth sqfr} always improves on the bound given by the main result of Balog \& Pomerance \cite{BaPom}. We remark that removing the squarefreeness condition does not help us to improve on Theorem~\ref{thm: smooth sqfr} due to the method used.

We also obtain a result for almost all primes. Firstly, we define the interval
$$
\cI(\beta )  = \(\alpha_0(\beta), \beta+1\)
$$
where 
$$
\alpha_0(\beta) =
\begin{cases}
5 (2- \beta )/3  & \text{ if $\beta \in (7/8, 13/14]$,} \\
 12-11\beta   & \text{ if $\beta \in (13/14, 17/18]$,} \\
(7- 4\beta)/2   & \text{ if $\beta \in (17/18, 25/26 ]$,} \\
  16-15\beta   & \text{ if $\beta \in (25/26, 31/32]$,} \\
 (18-11\beta)/5  & \text{ if $\beta \in (31/32, 41/42]$,} \\
 20-19\beta  & \text{ if $\beta \in (41/42, 49/50]$,} \\
(11-7\beta)/3  & \text{ if $\beta \in (49/50, 61/62]$,} \\
 24-23\beta  & \text{ if $\beta \in ( 61/62, 68/69]$,} \\
 4/3  &\text{ if $\beta \in (68/69,1]$.}
\end{cases}
$$

\begin{theorem} \label{thm:smooth sqfr almost all}
Fix real numbers $\alpha$ and $\beta$ such that $ \beta \in (7/8,1]$ and $\alpha \in \cI (\beta) $. Letting $x=Q^{\alpha +o(1)}$ and $y=Q^{\beta+o(1)}$,
as $Q \rightarrow \infty$, we have 
$$
\psi^\sharp(x,y;p,a) \ge x^{1+o(1)}/p
$$
 for all but $o(Q / \log Q)$ primes $p \in [Q, 2Q]$.
\end{theorem}

\subsection{Results for squarefull numbers}

First we observe that if $a$ is a quadratic residue modulo $p$ (or $a = 0$), then 
$a \equiv b^2 \bmod p$ for some integer $b \in [0, p-1]$ and so we have trivially
$F(a,p) \le (p-1)^2$ in this case. 

To estimate $F(a,p)$ for a  quadratic non-residue $a$ we denote
\begin{equation} \label{eq:Burg exp}
\eta_0 = \frac{1}{ 4\sqrt{e}}
\end{equation}
and recall that by the classical bound of Burgess~\cite{Burg} 
on the smallest quadratic non-residue $n_p$ we have  
\begin{equation}
 \label{eq:Burg}
n_p \le p^{\eta}
\end{equation}
 for any $\eta >\eta_0$ and a sufficiently large $p$. Noticing that $an_p^{-3}$ is a quadratic residue modulo $p$, we now obtain $F(a,p) \le n_p^3 (p-1)^2$. Hence, we have the trivial bound  $F(a,p) \le p^{2+3\eta_0 +o(1)} $ for any $a$, which we unfortunately do not know how to improve. However, we remark that assuming the Vinogradov's conjecture that $n_p \le p^{o(1)}$ (which is implied by the Generalised Riemann Hypothesis in the stronger form $n_p \ll \log^2 p$ proved by Ankeny \cite{Ankeny}, see also ~\cite[Section~5.9]{IwKow} for a discussion), we have the bound $F(a,p) \le p^{2+o(1)}$.   Even though we cannot reach such a bound, we obtain an unconditional better bound for almost all $a \in \{0, \ldots, p-1\}$.
 
We also note that  from a result on counting 
 squarefull integers~\cite{SS}, for any set $\cA$ of $A$ distinct residues modulo $p$ we have  
 $$
\max_{a \in \cA} F(a,p)  \gg A^2, 
$$
where, as usual, we use $A\ll B$ and $B \gg A$  as  an equivalent to the inequality $|A|\le cB$
with some  constant $c>0$, which occasionally, where obvious, may depend on
the real parameter $\varepsilon>0$. We slightly refine this result:

\begin{theorem} \label{thm:smooth sqfull almost all}
For all but $o(p)$ quadratic non-residues $a \in [0, p-1]$, we have 
$$
p^2n_p f(p) \ll F(a,p) \leq p^{2+\eta_0 +o(1)} 
$$
for any function $f(p)$ such that $f(p) \to 0$ as  $p \rightarrow \infty$
and 
$$
\max_{a \bmod p} F(a,p) \gg p^{2}n_p, 
$$
 where $n_p$ denotes the least quadratic non-residue modulo $p$.  
 \end{theorem}

 Using the lower bound in Theorem~\ref{thm:smooth sqfull almost all}, together with an unconditional
 result of Graham and Ringrose~\cite{GrRin} on primes with large values of $n_p$
 and a conditional result on the Generalised Riemann Hypothesis (GRH) of Montgomery~\cite{Montg}, 
 we immediately derive 
 
\begin{cor}\label{lowerboundnp}
For infinitely many primes $p$  we have 
$$\max_{a \bmod p} F(a,p)\gg \begin{cases}  p^{2} (\log p) (\log \log \log p), & \text{unconditionally,}\\
p^{2} (\log p)(\log \log p), & \text{under the GRH}.
 \end{cases}
$$
\end{cor}

In Section \ref{prepa}, we collect some results which will be used to prove the main results in Section \ref{proofs}.

\section{Preparation lemmas}\label{prepa}

\subsection{Exponential sums with reciprocals  of primes}
As usual, we define $\ep(z)= \exp(2i\pi z/p)$. For an integer $k$ with $\gcd(k,p)=1$ 
we use $\overline k$ to denote the multiplicative inverse of $k$ modulo $p$, 
that is, the unique integer with 
$$
k \overline k \equiv 1 \bmod p \mand 1 \le \overline k < p.
$$

 It is convenient to introduce the quantity
\begin{equation} 
\label{eq: BpL}
B(p,L) = \begin{cases}
L^{3/2} p^{1/8}, & \text{if} \ L < p^{1/3},\\
 L^{15/8} & \text{if} \  p^{1/3} \le  L <p.
\end{cases}
\end{equation}

The following bound of the double exponential sum over primes is a combination of~\cite[Lemma~3.5]{MuSh}
(for $L \le p^{1/3}$) and  of~\cite[Lemma~2.4]{Gar} (for $p^{1/3} \le  L <p$).

\begin{lemma}
\label{lem:BilinSums} For any real  $L \le p$, we have
$$
\max_{\gcd(a,p)=1} \left|\sum_{\ell_1,\ell_2 \in \cL} \ep\(a \overline \ell_1  \overline \ell_2\) \right| \le  B(p,L)  p^{o(1)},
$$
as $p\to \infty$, where $\cL$ is the set of primes $\ell \in [L,2L]$. 
\end{lemma}

\subsection{Some congruences with  products of primes}

We denote $N_{a,p}(L,h)$ to be the number of solutions to the congruence
\begin{equation} 
\label{equ: cong}
\ell_1 \ell_2 u \equiv a \bmod{p}, \qquad \ell_1 , \ell_2 \in \cL, \quad 1 \le u \le h,
\end{equation}
where $h$ and $L$ are two positive real numbers and $\cL$ is the set of primes $\ell \in [L, 2L].$ 

We now use Lemma~\ref{lem:BilinSums} to derive  an analogue of~\cite[Lemma~3.10]{MuSh}
(which also applies to $L \ge p^{1/3}$). 

\begin{lemma}
\label{lem:congr-asymp} For any integer $a$ and prime $p$ with $\gcd(a,p)=1$ and  reals $h$ and  $L$ 
with $1 \le h,L < p$, 
we have 
$$
N_{a,p}(L,h) = \frac{K^2 h}{p} + O\(B(p,L)  p^{o(1)}\), 
$$
where $K = \# \cL$ is the cardinality of $\cL$ and $B(p,L)$ is defined by~\eqref{eq: BpL}.
\end{lemma}

We also recall that by~\cite[Lemma~3.12]{MuSh} we have:

\begin{lemma} \label{lem: N bound}
For any integer $a$ and prime $p$ with $\gcd(a,p)=1$ and reals $1 \le h,L < p$ we have
$$
N_{a,p}(L,h) \le (L^2h/p+1)p^{o(1)}.
$$
\end{lemma}

Furthermore,  let $N^\sharp_{a,p}$ count the number of squarefree solutions to
the congruence~\eqref{equ: cong}. Following the proof of~\cite[Theorem~1.4]{MuSh}, 
but using a more general bound of Lemma~\ref{lem:congr-asymp} instead 
of~\cite[Lemma~3.10]{MuSh} as well as  Lemma~\ref{lem: N bound} (exactly as in~\cite{MuSh}), we derive

 \begin{lemma}
\label{lem:congr-asymp-sf} For any integer $a$ and prime $p$ with $\gcd(a,p)=1$ and reals $h$, $D$ and  $L$ with 
$$1 \le L,h < p \mand 1 \le D \le h^{1/2},
$$
 we have 
$$
N^\sharp_{a,p}(L,h) = \frac{K^2 h}{\zeta(2) p} +
O \( \(\frac{L^2 h}{Dp}+ DB(p,L)   +h^{1/2} \)p^{o(1)} \), 
$$
where $K = \# \cL$ is the cardinality of $\cL$ and $B(p,L)$ is defined by~\eqref{eq: BpL}.
\end{lemma}

We also need the bound of ~\cite[Lemma~3.14]{MuSh}  on the number of solutions  
$Q_{a,p}(L,h)$ to the congruence
$$
\ell_1 \ell_2^2 v \equiv a \bmod p, \qquad \ell_1,\ell_2 \in \cL, \ 1 \le v \le h.
$$

\begin{lemma}
\label{lem:congr-boundsquare} For any integer $a$ and prime $p$ with $\gcd(a,p)=1$ and reals $ 1\le L, h \le  p$ 
with $2L h \le p$
we have 
$$
Q_{a,p}(L,h) \le \(L h/p +1\)Lp^{o(1)}.
$$
\end{lemma}

It is shown in~\cite[Lemma~3.11]{MuSh}, that for almost all primes $p$, the asymptotic formula of  Lemma~\ref{lem:congr-asymp} can be improved as follows.

\begin{lemma} \label{lem: N bound almost all}
As $Q \rightarrow \infty$ for any fixed integer $k \ge 1$, for $1 \le L \le Q$ for all but $o(Q / \log Q)$ primes $p \in [Q, 2Q]$, for any integer $a$ with $\gcd(a,p)=1$ and real $h$ with $1 \le h \le p$, we have
$$
N_{a,p}(L,h) = \frac{K^2 h}{p} + O(  (L^{(3k-1)/2k} p^{1/2k}  + L^{(4k-1)/(2k)} )p^{o(1)} )
$$
where $K = \# \cL$ is the cardinality of $\cL$.
\end{lemma}

Finally, we also recall that by~\cite[Lemma~3.13]{MuSh} we have:

\begin{lemma} \label{lem: average N bound}
As $Q \rightarrow \infty$, for all but $o(Q/\log Q)$ primes $p \in [Q, 2Q]$, for any integer $a$, and reals $1\le F , L , h \le p$ with $F, L^2 h <p$, for the sum 
$$
R_{a,p}(F,L,h) = \sum_{F \le d \le 2F} N_{ad^{-2},p}(L,h)
$$
we have
$$
R_{a,p}(F,L,h) \le \max \{F(L^2 h)^{1/4} p^{-1/4}, F^{1/2} (L^2 h)^{1/4} \}p^{o(1)}.
$$
\end{lemma}

\subsection{Moments of character sums}
Let $\Omega_p$ denote the set of all  Dirichlet characters modulo $p$
and let  $\Omega_p^* = \Omega_p \setminus\{\chi_0\}$ denote the set of all 
non-principal Dirichlet characters modulo $p$.

We need the following result of 
Ayyad, Cochrane and Zheng~\cite[Theorem~2]{ACZ}, see also~\cite{Kerr} for a slightly 
sharper bound (which however does not change our final result).
 
\begin{lemma}
\label{lem:4th Mom}
 For any integer $K \geq 1$, we have
$$
\sum_{\chi \in \Omega_p^* } 
\left| \sum_{1 \le n \le K}
\chi(n)\right|^4 \le K^2p^{1+o(1)}. 
$$ 
\end{lemma}

\subsection{Quadratic non-residues in short intervals}

Let $T_p(K)$ denote the number of quadratic non-residues modulo $p$
in the interval $[1,K]$.

We need an extension of~\eqref{eq:Burg}. 
The following bound 
is given in~\cite[Theorem~2.1]{BHGS}.

\begin{lemma}
\label{lem:Nonres}
 For   any  real $\eta >\eta_0$, where $\eta_0$ is given by~\eqref{eq:Burg exp}, 
 there is a constant $c > 0$, such that for
 a sufficiently large $p$ and $K \ge p^\eta$ we have
$$
T_p(K) \ge c K. 
$$ 
\end{lemma}

\section{Proofs of main results}\label{proofs}

\subsection{Proof of Theorem~\ref{thm: smooth sqfr}}

For a sufficiently small $\varepsilon >0$, we set 
$$
L= p^{(\alpha - \beta )/2-\varepsilon/2}\mand  h= p^{\beta}.
$$
Since $L \le h \leq y$, $N^\sharp_{a,p}(L,h)$ counts a subset of $y$-smooth integers in an arithmetic progression.
Noticing that $L^2 h  \le  p^{\alpha -\varepsilon} = x^{1-\varepsilon +o(1)}$, we see that for a sufficiently large $p$ we have
\begin{equation}
 \label{eqn: lower bound} 
 \begin{split}
\psi^\sharp(x,y;p,a) & \ge N^\sharp_{a,p}(L,h)  + O\( \( h/p +1\)Lp^{o(1)}\)\\
& = N^\sharp_{a,p}(L,h)  + O\(Lp^{o(1)}\), 
\end{split}
\end{equation}
where we estimated the contribution coming from non-squarefree products $\ell_1 \ell_2 u$ (precisely products with $\ell_1 = \ell_2$ or with $ \ell_1 \mid u$ or with $ \ell_2\mid u$) using Lemma~\ref{lem:congr-boundsquare} with $h/L$ replacing $h$ as in the end of the proof of~\cite[Theorem~1.4]{MuSh}.

We use a crude estimate for the main term:
\begin{equation} \label{eq:MT}
\frac{K^2 h}{p} \sim  \frac{L^2 h}{p \(\log L\)^2} =  p^{\alpha -1 -\varepsilon + o(1)}.
\end{equation}

Choosing
$$
D=p^{\varepsilon /2}
$$ 
and using Lemma~\ref{lem:congr-asymp-sf}, we derive
\begin{equation}
\label{lowbound}
 \begin{split}
\psi^\sharp(x,y;p,a) & \gg \frac{K^2 h}{p} + O\(\( p^{\varepsilon /2}B(p,L) +L + h^{1/2}\) p^{o(1)}\)  \\ 
& = \frac{K^2 h}{p} + O\(\( p^{\varepsilon /2} B(p,L) + h^{1/2}\)  p^{o(1)}\)
\end{split}
\end{equation}
since $B(p,L)$ dominates $L$ and the main term~\eqref{eq:MT} dominates the first error term
 $L^2 h(Dp)^{-1}$ in  Lemma~\ref{lem:congr-asymp-sf}.

To begin, we remark that the term $h^{1/2}$ in~\eqref{lowbound} is dominated by the main term due to the inequality $\alpha-1 >9/2-3\beta-1  > \beta/2$ for $\beta \leq 1$.
We split the discussion on the contribution of $B(p,L)$ into two cases depending on $\alpha$.

Firstly, suppose that $\alpha \in (9/2 -3\beta, 2/3 +\beta]$. Since $\alpha \le 2/3+\beta$, this implies $L < p^{1/3}$ and hence $B(p,L) = L^{3/2}p^{1/8}$ by ~\eqref{eq: BpL}. Therefore, recalling~\eqref{lowbound} and ~\eqref{eq:MT}, we obtain
\begin{equation} \label{lower bound 1}
\psi^\sharp(x,y;p,a) \gg  p^{\alpha -1 -\varepsilon + o(1)} +O(p^{3(\alpha - \beta )/4 -\varepsilon/4 +1/8 +o(1)}).
\end{equation}
For $\varepsilon$ sufficiently small, we have $\alpha > 9/2 -3\beta +3\varepsilon$ which implies that the main term dominates trivially the remainder term in~\eqref{lower bound 1}.

Secondly, assume that $\alpha \in (2/3 +\beta, 3\beta]$. In particular, since $\beta \le 1$ we have
$$
2/3+\beta +\varepsilon \le \alpha <3\beta < 2 + \beta +\varepsilon,
$$
for $\varepsilon >0$ chosen sufficiently small. Hence $p^{1/3} \le L < p$ and we have $B(p,L)=L^{15/8}$ by ~\eqref{eq: BpL}. Therefore, recalling ~\eqref{lowbound} and ~\eqref{eq:MT}, we obtain
\begin{equation} \label{lower bound 2}
 \psi^\sharp(x,y;p,a) \ge    p^{\alpha -1 -\varepsilon + o(1)}  + O(p^{15(\alpha - \beta)/16 -7\varepsilon/16 +o(1)}).
\end{equation}
 Notice that we have
$$
\alpha > 2/3 +\beta \ge 16-15\beta +9\varepsilon +o(1)
$$
 when $\beta \in (23/24,1]$ and $\varepsilon >0$ is sufficiently small. It follows that the main term dominates the remainder term in~\eqref{lower bound 2}. Therefore, in all cases we conclude
$$
\psi^\sharp(x,y;p,a) \ge p^{\alpha -\varepsilon -1 +o(1)}.
$$
Since this is valid for all sufficiently small $\varepsilon >0$, the result follows.

\subsection{Proof of Theorem~\ref{thm:smooth sqfr almost all}}

We follow the proof of~\cite[Theorem 1.6]{MuSh}. 
For $\varepsilon >0$, we set  
\begin{equation} \label{equ: LhDE}
L=Q^{ (\alpha - \beta)/2 - \varepsilon/2 }, \quad h = Q^{\beta}, \quad D=Q^{\varepsilon/2}, \quad 
E=Q^{(\alpha-1)/2}.
\end{equation}

 We note that $(\alpha-1)/2> 0$ for $\alpha \in \cI(\beta )$ and so $D < E$ if $\varepsilon >0$ is sufficiently small. We also have $E < h^{1/2}$ since $\alpha < \beta +1$.

Since $\alpha < \beta +1 \le 3\beta$ in the range $\beta \in (7/8,1]$, we get $L \le h$. In 
particular, we have as before the inequality ~\eqref{eqn: lower bound}. 

By inclusion and exclusion, we have
\begin{equation} \label{eq:N Sigma}
N^\sharp_{a,p}(L,h)   = \sum_{d \le h^{1/2}} \mu(d) N_{ad^{-2},p}(L, h/d^2) 
 = \Sigma_1 + \Sigma_2 + \Sigma_3,
\end{equation}
where
\begin{align*}
& \Sigma_1 = \sum_{d \le D} \mu(d) N_{ad^{-2},p}  (L,h/d^2),\\
&\Sigma_2 = \sum_{D <d  \le E} \mu(d) N_{ad^{-2},p}  (L,h/d^2),\\
&\Sigma_3 = \sum_{E < d  \le h^{1/2}} \mu(d) N_{ad^{-2},p}  (L,h/d^2).
\end{align*}

To abstain from clutter, all the bounds below are valid for all but $o(Q/\log Q)$ primes  $p \in [Q,2Q]$.

Since $\alpha < \beta +1 < 2+\beta +\varepsilon +o(1)$ and $\beta \le 1$, we obtain respectively $L \le Q$ and $h \le p$. By Lemma~\ref{lem: N bound almost all}
\begin{equation} \label{eq:Sigma1}
\begin{split}
\Sigma_1 & =  \frac{K^2 h}{\zeta(2) p} \\
& \qquad +  O \( \frac{K^2 h}{Dp}+  D(L^{(3k-1)/(2k)}p^{1/(2k)} +L^{(4k-1)/(2k)})p^{o(1)}   \)
\end{split}
\end{equation}
for any fixed positive integer $k$.

By Lemma~\ref{lem: N bound} with $h/d^2$ replacing $h$ there, we have 
\begin{equation} \label{eq:Sigma2}
\Sigma_2 \le \(\frac{L^2 h}{Dp} +E \)p^{o(1)}.
\end{equation}

We split $\Sigma_3$ into $O(\log p)$ sums with intervals of the form $[F,2F]$ where $E \le F \le h^{1/2}$. 

From the choice of $E$ in~\eqref{equ: LhDE} we see  that 
$$L^{2}h/d^2 \le L^{2}h/F^2  \le L^{2}h/E^2  < p,
$$ 
hence by Lemma~\ref{lem: average N bound}
\begin{align*}
R_{a,p}(F,L,h/F^2) & \le \max \{ F(L^2 h/F^2)^{1/4} p^{-1/4}, F^{1/2}(L^2 h/F^2)^{1/4} \} p^{o(1)} \\
& = (L^2 h)^{1/4} p^{o(1)}
\end{align*}
since $F \le h^{1/2} \le p^{1/2}$ and so
\begin{equation} \label{eq:Sigma3}
\Sigma_3 \le (L^2 h)^{1/4} p^{o(1)}.
\end{equation}

Substituting~\eqref{eq:Sigma1}, \eqref{eq:Sigma2} and~\eqref{eq:Sigma3} in~\eqref{eq:N Sigma}, 
we obtain
$$
N^\sharp_{a,p}(L,h) = \frac{K^2 h}{\zeta(2) p} +  O\(Rp^{o(1)}  \)
$$ where we set 
$$
R =  D\(L^{(3k-1)/(2k)}p^{1/(2k)}  +L^{(4k-1)/(2k)} \) + \frac{L^2 h}{Dp} +E + (L^2 h)^{1/4}    
$$
and the main term verifies an analogue of~\eqref{eq:MT}, precisely, 
\begin{equation} \label{eq:MT-Q}
\frac{K^2 h}{p} \sim  \frac{L^2 h}{p \(\log L\)^2} =  Q^{\alpha -1 -\varepsilon + o(1)}.
\end{equation}

Notice that  the choice of $E$ in~\eqref{equ: LhDE} implies that $E$ is smaller than the main term~\eqref{eq:MT-Q}. We now see  from~\eqref{eq:MT-Q} that if
\begin{equation} 
 \label{equ: alpha condition}
\alpha -1  > \max \left\{   \frac{3k-1}{2k} \frac{\alpha - \beta}{2} + \frac{1}{2k},    \frac{4k-1}{2k}  \frac{\alpha - \beta}{2} , 
\frac{\alpha}{4}  , \frac{\alpha - \beta}{2} \right \} 
\end{equation}
for some positive integer $k$, then for a sufficiently small $\varepsilon$ the main term dominates the remainder term in~\eqref{eqn: lower bound} and the result follows. 

Rearranging~\eqref{equ: alpha condition} gives
\begin{equation} 
 \label{equ: alpha condition 2}
\alpha > \max \left  \{ \frac{(1-3k)\beta + 2 +4k}{k+1} , (1-4k)\beta +4k , 4/3, 2-\beta \right \}.
\end{equation}

 First, we remark that $2-\beta \leq (1-4k)\beta +4k $ since $\beta \leq 1$ and we can discard $2-\beta$ from the maximum in~\eqref{equ: alpha condition 2}. 
 
Furthermore, for $k\leq 5$, we see that $4/3$ is dominated by the first term of the right hand side of~\eqref{equ: alpha condition 2}. In this case, a quick computation shows that 
$$
\frac{(1-3k)\beta + 2 +4k}{k+1} \geq  (1-4k)\beta +4k
$$ 
if and only if $\beta \geq 1-1/2k^2$. 
Thus, in the interval $(1-1/2k^2,1-1/2(k+1)^2]$, the maximum is given either by $\((1-3k)\beta + 2 +4k\)/{(k+1)}$ 

or by $(1-4m)\beta +4m$ with $m\geq k+1$. Since the function $f(z)=(1-4z)\beta +4z$ is a 
monotonically increasing function of $z$, we check only the case $m=k+1$ and verify that 
$$\frac{(1-3k)\beta + 2 +4k}{k+1} \geq f(k+1) = (1-4(k+1))\beta +4(k+1)
$$
 if and only if 
$$
\beta\ge \beta_0(k)
$$
where 
$$
\beta_0(k) = 1-\frac{1}{2(k^2+k+1)}.
$$
Splitting the interval 
$$
\cI_k = \(1-\frac{1}{2k^2},1-\frac{1}{2(k+1)^2} \right]
$$ into two intervals as follows
\begin{align*}
\cI_k& = \(1-\frac{1}{2k^2},1-\frac{1}{2(k^2+k+1)}\right]\\
& \qquad \qquad \qquad \qquad \bigcup  \(1-\frac{1}{2(k^2+k+1)},1-\frac{1}{2(k+1)^2}\right]
\end{align*}
and recalling that $k \le 5$, we deduce after short computations the result for $\beta \leq  \beta_0(5) = 61/62$. 
 
 For $k\geq 6$, noticing that $(1-4\beta)+4k \geq 4/3$ in the range 
$$
\beta \leq 1-\frac{1}{3(4k-1)},
$$ we also deduce the case $\beta \in (61/62, 68/69]$. 

For the remaining case $\beta \in ( 68/69, 1]$ and $k\geq 6$, we see that 
$$
\frac{(1-3k)\beta + 2 +4k}{k+1}\leq 4/3.
$$


Based on the above argument, we now give  explicit choices of $k$ and corresponding 
intervals which optimise our bound. 

\begin{itemize}
\item If $\beta \in (7/8, 13/14]$,   we take $k=2$ and~\eqref{equ: alpha condition 2} simplifies to 
$$
\alpha  > \max \left \{ 5(2-\beta)/3, 8-7\beta,  4/3, 2-\beta \right  \}  = 5(2-\beta)/3.
$$ 

\item If $\beta \in (13/14 , 17/18  ]$,    we take $k=3$ and~\eqref{equ: alpha condition 2} simplifies to
$$
\alpha > \max \left  \{  (7-4 \beta)/2, 12 -11\beta,4/3, 2-\beta  \right \} = 12 -11\beta.
$$

\item If $\beta \in (17/18, 25/26]$,  we take $k=3$ and~\eqref{equ: alpha condition 2} simplifies to
$$
\alpha > \max \left  \{(7-4 \beta)/2, 12 -11\beta,4/3, 2-\beta  \right \} = (7-4 \beta)/2.
$$

\item If $\beta \in (25/26,31/32 ]$,   we take $k=4$ and~\eqref{equ: alpha condition 2} simplifies to
$$
\alpha >  \max \left \{  (18-11\beta)/5, 16-15\beta, 4/3, 2-\beta  \right \}  = 16-15\beta.
$$

\item If $\beta \in (31/32,41/42 ]$,  we take $k=4$ and~\eqref{equ: alpha condition 2} simplifies to
$$
\alpha >  \max \left \{ (18-11\beta)/5, 16-15\beta, 4/3, 2-\beta  \right \}  = (18-11\beta)/5.
$$

\item If $\beta \in (41/42,49/50 ]$,   we take $k=5$ and \eqref{equ: alpha condition 2} simplifies to
$$
\alpha >  \max \left \{ (11-7\beta)/3 , 20-19\beta, 4/3, 2-\beta \right \} = 20-19\beta.
$$

\item If $\beta \in (49/50, 61/62 ]$,   we take $k=5$ and \eqref{equ: alpha condition 2} simplifies to
$$
\alpha >  \max \left \{ (11-7\beta)/3, 20-19\beta, 4/3, 2-\beta \right \} = (11-7\beta)/3.
$$

\item If $\beta \in (61/62, 68/69]$,   we take $k=6$ and \eqref{equ: alpha condition 2} simplifies to
$$
\alpha >  \max \left \{  (26-17 \beta)/7, 24-23\beta, 4/3, 2-\beta \right \}=  24-23\beta.
$$

\item If $\beta \in (68/69, 1]$,   we take $k=6$ and \eqref{equ: alpha condition 2} simplifies to
$$
\alpha >  \max \left \{(26-17 \beta)/7, 24-23\beta, 4/3, 2-\beta \right \}=  4/3.
$$
\end{itemize}

Therefore in all cases, where we also recall the condition $ \alpha < \beta +1$,  we have
$$
\psi^\sharp(x,y;p,a) \ge  p^{\alpha -1 -\varepsilon+o(1)}.
$$ 
Since this is true for all $\varepsilon >0$, the result follows immediately.

\subsection{Proof of Theorem~\ref{thm:smooth sqfull almost all}}
Let $M$ be a parameter which will be fixed later. We introduce the subset of residues modulo $p$
$$ 
\cS=\left\{a: ~ a \textrm{ quadratic non-residue  such that } F(a,p) \leq M \right\} .
$$ 
Firstly, we remark that every squarefull integer $n$ can be written as $n=r^2s$ with $s\mid r$. Furthermore, if $a$ is a quadratic non-residue, we notice that $s$ has to be a quadratic non-residue in this representation; in particular $s \ge n_p$. 

Let us count the number of products $r^2s \leq M$ with $s\mid r$ and $s \geq n_p$ the smallest quadratic non-residue modulo $p$. Noticing that $r \le (M/s)^{1/2}$, we have at most $M^{1/2}s^{-3/2}$ possible values of $r$. Thus the number of different products $r^2s$ is bounded by $$ \sum_{s\geq n_p} M^{1/2}s^{-3/2} \ll  M^{1/2}n_p^{-1/2}.  $$ This implies

$$ \# \cS \ll M^{1/2}n_p^{-1/2}.$$ Setting $M = p^{2} n_p f(p)$, we get $ \# \cS = o(p)$ which concludes the proof. The assertion $\max_{a \bmod p} F(a,p) \gg p^{2}n_p$ follows by the same argument
by setting
$$ 
\cS=\left\{a:~ a\textrm{ quadratic non-residue} \right\} \quad \text{and}\quad  M = \max_{a\in \cS} F(a,p).
$$ 
 So we now turn our attention to the upper bound. 

Clearly, if $a\equiv u^2 \mod p$, $0 \le u < p$,  is quadratic residue (or $a = 0$), 
then  $F(a,p) \le u^2 \le p^2$.

We now fix some $\varepsilon > 0$ and denote by $\cA$ the set of quadratic non-residues 
for which $F(a,p) \ge p^{2+\eta_0 +\varepsilon} $.

It is enough to show that  the cardinality of $\cA$ satisfies 
\begin{equation} 
\label{eq: A small} 
\# \cA = o(p).
\end{equation}

We set 
$$
K =  \rf{p^{\eta_0 +\varepsilon/2}} \mand U =  \rf{p^{1-\eta_0}}
$$ 
and let $\cN$ be the set of quadratic  non-residues in the interval $[1,K]$.  In particular
$$
\# \cN = T(K).
$$

Clearly for $a\in \cA$ the congruence 
\begin{equation} \label{eqn: an3u2} 
a \equiv n^3 u^2 \mod p, \hspace{2mm} n \in \cN, \ 1 \le u \le U, 
\end{equation}
has no solution. Thus expressing the number of solutions to~\eqref{eqn: an3u2} via characters we see that 
$$
\sum_{n \in \cN} \sum_{1\leq u \leq  U} \frac{1}{p-1} \sum_{\chi \in \Omega_p} \chi\(a^{-1} n^{3} u^{2}\)  = 0.
$$
Summing over all $a \in \cA$ and using the multiplicativity of characters, we arrive to 
\begin{equation} 
\label{eqn: vanish} 
 \sum_{\chi \in \Omega_p}  \sum_{a \in \cA} \overline \chi\(a\)
\sum_{n \in \cN} \chi\(n\)^{3}  \sum_{u=1}^U \chi\(u\)^{2}  = 0 , 
\end{equation}
where $ \overline \chi$ denotes the complex conjugate character of $\chi$.

Now, the contribution to~\eqref{eqn: vanish}  from the principal character is
obviously $\#\cA T(K) U$.

Furthermore, since all elements of $\cA$ are quadratic non-residues, the contribution
to~\eqref{eqn: vanish} from the quadratic character, that is, from the Legendre symbol is 
\begin{align*}
  \sum_{a \in \cA} \(\frac{a}{p}\)
\sum_{n \in \cN} \(\frac{n}{p}\)^{3}  \sum_{u=1}^U  \(\frac{u}{p}\)^{2} 
& =   \sum_{a \in \cA} (-1)
\sum_{n \in \cN} (-1)^{3}  \sum_{u=1}^U 1\\
&  = \#\cA T(K) U. 
\end{align*}

This allows us to write~\eqref{eqn: vanish} as
\begin{equation} 
\label{eqn: MT} 
2 \#\cA T(K) U = - \sum_{\chi \in \Omega_p^\sharp}  \sum_{a \in \cA} \overline \chi\(a\)
\sum_{n \in \cN} \chi\(n\)^{3}  \sum_{u=1}^U \chi\(u\)^{2} 
\end{equation}
with $\Omega_p^\sharp$ being the subset of $\Omega_p^{*}$ where we removed the quadratic character.

Now, for  $\chi \in \Omega_p^\sharp$ we have by definition $\chi^2 \ne \chi_0$. Furthermore, each character from $\Omega_p^*$ occurs at most twice as $\chi^2$ and each character from $\Omega_p$ (including also $\chi_0$ in this case) occurs at most three times as $\chi^3$ for  $\chi \in \Omega_p^\sharp$. 

Using the H{\"o}lder inequality, we now derive from~\eqref{eqn: MT}  that 
\begin{equation} 
\label{eq:Holder} 
2 \#\cA T(K) U \le   \Sigma_1^{1/2} \Sigma_2^{1/4}  \Sigma_3^{1/4}
\end{equation}
where
\begin{align*}
& \Sigma_1 =  \sum_{\chi \in \Omega_p^\sharp}  \left| \sum_{a \in \cA} \overline \chi\(a\)\right|^2 \le
 \sum_{\chi \in \Omega_p}  \left| \sum_{a \in \cA}   \chi\(a\)\right|^2,\\
& \Sigma_2 =  \sum_{\chi \in \Omega_p^\sharp}  \left| \sum_{n \in \cN}  \chi\(n\)^3\right|^4
\le  3 \sum_{\chi \in \Omega_p}  \left| \sum_{n \in \cN}  \chi\(n\)\right|^4,\\
& \Sigma_3 =  \sum_{\chi \in \Omega_p^\sharp }  \left|   \sum_{u=1}^U \chi\(u\)^{2} \right|^4 
\le   2 \sum_{\chi \in \Omega_p^* }  \left|   \sum_{u=1}^U \chi\(u\) \right|^4, 
\end{align*} and the upper bounds come from the discussion above. We now see by the orthogonality of characters that we have
\begin{equation} 
\label{eq:SigmaA}
\Sigma_1 \le (p-1)\#\cA.
\end{equation} For $\Sigma_2$,  using again the orthogonality of characters, we write 
\begin{align*} 
\Sigma_2 & =  3 (p-1) \#  \left\{(n_1,n_2,n_3,n_4) \in \cN:~n_1n_2 \equiv n_3 n_4 \bmod p\right\}\\
& \le  3 (p-1) \#  \left\{(n_1,n_2,n_3,n_4) \in [1,K]:~n_1n_2 \equiv n_3 n_4 \bmod p\right\}\\
& = 3 \sum_{\chi \in \Omega_p }  \left|   \sum_{n=1}^K \chi\(n\) \right|^4 =
 3 K^4 +  \sum_{\chi \in \Omega_p^* }  \left|   \sum_{n=1}^K \chi\(n\) \right|^4. 
\end{align*}
Applying Lemma~\ref{lem:4th Mom} and using that $K^2 \le p$ provided that $\varepsilon$ is 
small enough, we derive 
\begin{equation} 
\label{eq:SigmaN}
\Sigma_2 \le  K^2p^{1+o(1)}.
\end{equation} Finally, we also estimate $\Sigma_3$, directly by Lemma~\ref{lem:4th Mom} 
getting
\begin{equation} 
\label{eq:SigmaU}
\Sigma_3 \le  U^2p^{1+o(1)}.
\end{equation}

Substituting~\eqref{eq:SigmaA}, \eqref{eq:SigmaN} and~\eqref{eq:SigmaU} in~\eqref{eq:Holder}, 
we now derive
$$
 \#\cA T(K) U \le  ( \#\cA)^{1/2} K^{1/2}  U^{1/2} p^{1+o(1)}
$$
which together with  Lemma~\ref{lem:Nonres}  yields
$$
 \#\cA  \le  K^{-1}  U^{-1} p^{2+o(1)} = p^{1-\varepsilon/2+o(1)}.
$$
We now see that~\eqref{eq: A small}  holds which concludes the proof. 

\appendix

\section{Short Intervals with Many $k$-full Numbers}
\label{app:A}

\begin{theorem}
\label{thm:Main}  For any fixed integer $k\ge 2$, there are infinitely
many $N$, such that the open interval $(N^k,(N+1)^k)$ contains
at least
$$
M \ge \sqrt{\(\frac{k}{2(k+1)}  + o(1)\)  \frac{\log N}{\log \log N}}
$$
$k$-full integers.
\end{theorem}

    \begin{proof}  Let $1 < d_1 < \ldots < d_{2\ell}$ be the first $2\ell$
squarefree integers greater than $1$,
that is,  $d_j = \pi^2 j/6 + o(j)$ 
and remark that $d_j\le 4\ell$, provided that $\ell$
is large enough.

Let $\alpha_j = d_j^{-(k+1)/k}$, $j =1, \ldots, 2\ell$. We define
$$
R =  \(k 2^{k-1} (4\ell)^{(k+1)/k} \)^{2\ell}
$$
and let $q$ be the smallest integer
$$
q \ge R
$$
for which, for some integers $r_j$,
\begin{equation}
\label{eq:Approx}
    \left| \alpha_j - \frac{r_j}{q} \right| \le \frac{1}{q^{1+1/2\ell}},
\qquad j =1, \ldots, 2\ell.
\end{equation}
We see that
$$
q \le Q,
$$
where  $Q=C(\alpha_1,\delta)^{-2\ell} R^{2\ell (1 + \delta)}$ for some constant 
$C(\alpha_1,\delta)> 0$ depending only on $\alpha_1$ and $\delta$. 
Indeed, otherwise applying the {\it Dirichlet Approximation
Theorem\/}, see~\cite[Corollary 1B, p. 27]{Schmidt}, we conclude  that
$$
\left| \alpha_j - \frac{r_j}{s} \right|
    \le \frac{1}{s Q^{1/2\ell}} ,
\qquad j =1, \ldots, 2\ell,
$$
for some positive integer $s\le Q$. Due to the minimality
condition on  $q$, we have $s \le R$. On the other hand,
by the {\it Roth  Theorem\/}, see~\cite[Theorem 2A, p. 116]{Schmidt},  we have
$$
\frac{C(\alpha_1,\delta)} {s^{2+\delta}} <
\left| \alpha_1 - \frac{r_1}{s} \right|
\le \frac{1}{s Q^{1/2\ell}}.
$$
Therefore
$$
s >
\(C(\alpha_1,\delta)Q^{1/2\ell}\)^{1/(1+\delta)} = R,
$$
   which is impossible.

We see from~\eqref{eq:Approx} that, for $j =1, \ldots, 2\ell$,
$$
|q - d_j^{(k+1)/k} r_j |
    \le \frac{(4\ell)^{(k+1)/k}}{q^{1/2\ell}}
\le 1
$$
(provided that $\ell$ is large enough. 
Therefore,
$$
d_j^{(k+1)/k}r_j = \alpha_j^{-1} r_j \le q+1,
\qquad j =1, \ldots, 2\ell.
$$
We now derive,
\begin{align*}
|q^k - d_j^{k+1} r_j^k| & =   
|q - d_j^{(k+1)/k} r_j |  \sum_{\nu =0}^{k-1} q^{k-1-\nu}
(d_j^{(k+1)/k}r_j )^\nu\\
& \le  k (q+1)^{k-1} |q - d_j^{(k+1)/k} r_j | 
 \le   \frac{k (4\ell)^{(k+1)/k} (q+1)^{k-1}} {q^{1/2\ell}}\\
&   < \frac{k 2^{k-1} (4\ell)^{(k+1)/k} } {q^{1/2\ell}} q^{k-1} \le \frac{k 2^{k-1} (4\ell)^{(k+1)/k} } {R^{1/2\ell}} q^{k-1}.
\end{align*}

Recalling the choice of $R$ we now see that 
$$
|q^k - d_j^{k+1} r_j^k| < q^{k-1}.
$$
Therefore one of the intervals $((q-1)^k, q^k)$
or $(q^k,(q+1)^k)$ contains at least $M \ge \ell $ of the integers
$d_j^{k+1} r_j^k$, $j =1, \ldots, 2\ell$,
which are obviously pairwise distinct (because $d_j$ is squarefree
for all $j=1,\dots,2\ell$), and $k$-full.
We now have
\begin{align*}
q &  \le  Q =  C(\alpha_1,\delta)^{-2\ell}R^{2\ell (1 + \delta)}  \\
& \le  
C(\alpha_1,\delta)^{-2\ell}  \(k 2^{k-1} 
(4\ell)^{(k+1)/k} \)^{4\ell^2 (1+\delta)} \\
   & =  \exp\( \(4(k+1)k^{-1}(1+\delta) + o(1)\) \ell^2  \log \ell\).
\end{align*}
Hence, since $k$ is fixed,
$$
\ell^2 \log \ell \ge \(\frac{k}{4(k+1)(1+\delta)} + o(1)\) \log q.
$$
In particular, considering two cases  
$$
\log \ell \ge \frac{1}{2} \log \log q  \mand  \log \ell <  \frac{1}{2} \log \log q
$$
this implies that
$$
M \ge \ell \ge \(\(\frac{k}{2(k+1)(1+\delta)}\)^{1/2} + o(1)\) \(\frac{\log q}{\log \log
q}\)^{1/2}.
$$
Recalling that $\delta$ is arbitrary,  the proof is complete. 
\end{proof}

%
%
%
%

\section*{Acknowledgement}

The authors are grateful to V.~Blomer for giving them the idea of the argument in Appendix~\ref{app:A}. 

 This work  was partially supported by the Austrian Science Fund (FWF) project Y-901 (for M.M.), by the  Australian Research Council (ARC)  Grant DP170100786 (for I.S.) and by an Australian Government Research Training Program (RTP) Scholarship (for K.H.Y.).

\end{document}